\documentclass[12pt]{article}
\usepackage{amsmath}
\usepackage{amssymb}
\usepackage{amsthm}

\usepackage[margin=2.0cm]{geometry}
\usepackage{hyperref}
\usepackage{multirow}

\usepackage{color}
\usepackage{enumerate}

\usepackage{tikz,lmodern}
\usepackage{amsfonts}
\usetikzlibrary{arrows, decorations.markings}

\newtheorem{lemma}{Lemma}[section]%
\newtheorem{theorem}[lemma]{Theorem}%
\newtheorem{proposition}[lemma]{Proposition}%
\newtheorem{corollary}[lemma]{Corollary}%

\def\A{\hbox{\rm A}}
\def\S{\hbox{\rm S}}
\def\PSL{\hbox{\rm PSL}}
\def\PSp{\hbox{\rm PSp}}
\def\PSU{\hbox{\rm PSU}}
\def\M{\hbox{\rm M}}
\def\P{\hbox{\rm P}}
\def\SL{\hbox{\rm SL}}
\def\Sp{\hbox{\rm Sp}}
\def\AGL{\hbox{\rm AGL}}

\def\J{\hbox{\rm J}}
\def\PGO{\hbox{\rm PGO}}
\def\GL{\hbox{\rm GL}}
\def\PGL{\hbox{\rm PGL}}
\def\Aut{\hbox{\rm Aut}}
\def\Inn{\hbox{\rm Inn}}
\def\Cay{\hbox{\rm Cay}}
\def\F{\hbox{\rm F}}
\def\Out{\hbox{\rm Out}}
\def\soc{\hbox{\rm soc}}

\numberwithin{equation}{section}
\def\mz{{\mathbb Z}}

\begin{document}
\title{Symmetric Cayley graphs on non-abelian simple groups of valency $7$}
\author{\\Xing Zhang, Yan-Quan Feng, Fu-Gang Yin\footnotemark, Hong Wang\\
{\small\em School of mathematics and statistics, Beijing Jiaotong University, Beijing, 100044, P.R. China}\\
}

\renewcommand{\thefootnote}{\fnsymbol{footnote}}
\footnotetext[1]{Corresponding author.
E-mails:
21118006@bjtu.edu.cn (X. Zhang),
yqfeng@bjtu.edu.cn (Y.-Q. Feng),
fgyin@bjtu.edu.cn (F.-G. Yin),
13522941618@163.com (H. Wang)}

\date{}
\maketitle

\begin{abstract}
Let $\Gamma$ be a connected $7$-valent symmetric Cayley graph on a finite non-abelian simple group $G$. If $\Gamma$ is not normal, Li {\em et al.} [On 7-valent symmetric Cayley graphs of finite simple groups, J. Algebraic Combin. 56 (2022) 1097-1118] characterised the group pairs $(\mathrm{soc}(\mathrm{Aut}(\Gamma)/K),GK/K)$, where $K$ is a maximal intransitive normal subgroup of $\mathrm{Aut}(\Gamma)$. In this paper, we improve this result by proving that if $\Gamma$ is not normal, then $\mathrm{Aut}(\Gamma)$ contains an arc-transitive non-abelian simple normal subgroup $T$ such that $G<T$ and $(T,G)=(\mathrm{A}_{n},\mathrm{A}_{n-1})$ with $n=7$, $3\cdot 7$, $3^2\cdot 7$, $2^2\cdot 3\cdot 7$, $2^3\cdot3\cdot7$, $2^3\cdot3^2\cdot5\cdot7$, $2^4\cdot3^2\cdot5\cdot7$, $2^6\cdot3\cdot7$, $2^7\cdot3\cdot7$, $2^6\cdot3^2\cdot7$, $2^6\cdot3^4\cdot5^2\cdot7$, $2^8\cdot3^4\cdot5^2\cdot7$, $2^7\cdot3^4\cdot5^2\cdot7$, $2^{10}\cdot3^2\cdot7$, $2^{24}\cdot3^2\cdot7$. Furthermore, $\mathrm{soc}(\mathrm{Aut}(\Gamma)/R)=(T\times R)/R$, where $R$ is the largest solvable normal subgroup of $\mathrm{Aut}(\Gamma)$.
\bigskip

\end{abstract}

\section{Introduction}

For a graph $\Gamma$, denote by $V(\Gamma)$, $E(\Gamma)$ and $\Aut(\Gamma)$ the vertex set, edge set and automorphism group of $\Gamma$, respectively. Let $X\leq \Aut(\Gamma)$. The graph $\Gamma$ is called \emph{$X$-vertex-transitive} or \emph{$X$-arc-transitive} if $X$ is  transitive on $V(\Gamma)$ or on the arc set of $\Gamma$, where an arc is an ordered pair of adjacent vertices. Furthermore, $X$-arc-transitive is also called {\em $X$-symmetry}, and $\Gamma$ is called {\em symmetry} or {\em arc-transitive} if it is $\Aut(\Gamma)$-symmetry.
The group $X$ is called  \emph{semiregular} if for every $v\in V(\Gamma)$, the stabilizer $X_v$ (the subgroup of $X$ fixing $v$) is trivial, and  \emph{regular} if $X$ is transitive and semiregular. Throughout this paper, graph is connected, finite and simple, and group is finite.

Let $G$ be a finite group and $S$ a subset of $G$ such that $1\not\in S$ and $S^{-1}=S$.
The \emph{Cayley graph} $\Cay(G, S)$ of $G$ is defined as the graph with vertex set $G$ and  two vertices $x$ and $y$ being adjacent if and only if $yx^{-1}\in S$. It is easy to see that $\Cay(G, S)$ is connected if and only if $S$ generates group $G$, and a graph $\Gamma$ is isomorphic to a Cayley graph of $G$ if and only if $\Gamma$ has a regular group of automorphisms isomorphic to $G$.

Let $\Gamma=\Cay(G,S)$. For $g\in G$, define $R(g): x\mapsto xg$ for all $x\in G$. Then $R(g)\leq \Aut(\Gamma)$ and $R(G):=\{R(g)\ |\ g\in G\}$ is regular subgroup of $\Aut(\Gamma)$, called the right regular representation of $G$. The graph $\Gamma$ is called \emph{normal} if $R(G)\unlhd \Aut(\Gamma)$, that is, $R(G)$ is a normal subgroup of $\Aut(\Gamma)$. Set $\Aut(G,S)=\{\alpha\in\Aut(G)\ |\ S^\alpha=S\}$. It is easy to see that $R(g)^\alpha=R(g^\alpha)$ for every $\alpha\in \Aut(G,S)$ and every $g\in G$, which implies that $R(G)$ is normalized by $\Aut(G,S)$. Since $R(G)$ is regular and $\Aut(G,S)$ fixes $1$, we have $R(G)\Aut(G,S)=R(G)\rtimes \Aut(G,S)$. By  Godsil~\cite{Godsil}, the normalizer $N_{\Aut(\Gamma)}(R(G))$ of $R(G)$ in $\Aut(\Gamma)$ is $R(G)\rtimes \Aut(G,S)$. It follows that if $\Gamma$ is normal, then based on $G$ and $\Aut(G,S)$, $\Aut(\Gamma)$ is explicitly known, that is, $R(G)\rtimes \Aut(G,S)$. For this reason, to investigate Cayley graphs, we mainly concern non-normal ones in most cases.

In this paper, we study connected arc-transitive Cayley graph $\Gamma=\Cay(G,S)$ on a  non-abelian simple group $G$, which is currently a hot topic in algebraic graph theory. Let $\Gamma$ be non-normal and of valency $3$. In 1996, Li~\cite{Li} proved that $G=\A_5$, $\A_7$, $\PSL(2,11)$, $\M_{11}$, $\A_{11}$, $\A_{15}$, $\M_{23}$, $\A_{23}$ or $\A_{47}$. This was refined in 2005 by Xu {\em et al.}~\cite{XFWX1} that $G= \A_{47}$, and two years later, it was further refined in \cite{XFWX2} that for $G=\A_{47}$, $\Gamma$ must be $5$-arc-transitive and there are exactly two such graphs up to isomorphism, which is a remarkable achievement on this line.

Let $\Gamma$ be non-normal and of valency $5$. It follows from Fang~{\em et al.}~\cite{FMW} that if $\Aut(\Gamma)$ is quasiprimitive then $(G,\mathrm{soc}({\rm Aut(\Gamma))})=(\A_{n-1}, \A_n)$, where either $n=60\cdot k$ with $k \ |\ 2^{15}\cdot3$ and $k \neq 3, 4, 6, 8$, or $n = 10\cdot m$ with $m\ |\ 8$, and that if $\Aut(\Gamma)$ is not quasiprimitive then there is a maximal intransitive normal subgroup $K$ of $\Aut(\Gamma)$  such that the socle of $\Aut(\Gamma)/K $, denoted by $\bar{L}$, is a simple group containing $\bar{G}=GK/K\cong G$ properly, where either $(\bar{G},\bar{L})=(\Omega^-(8,2),\rm \PSp(8,2))$, $(\A_{14},\A_{16})$, $(\A_{n-1},\A_n)$ with $n\geq 6$ such that $n$ is a divisor of $2^{17}\cdot 3^2\cdot 5$, or $\bar{G}=\bar{L}$ is isomorphic to an irreducible subgroup of $\PSL(d,2)$ for $4\leq d \leq 17$ and the $2$-part $|G|_2>2^d$. This was refined by Du {\em et al.}~\cite{DFZh} that $\Aut(\Gamma)$ contains a non-abelian simple normal subgroup $T$ such that $G\leq T$ and $(G,T)=(\A_{n-1},\A_n)$ with $n=2\cdot 3$, $2^2\cdot 3$, $2^4$, $2^3\cdot 3$, $2^5$, $2^2\cdot 3^2$, $2^4\cdot 3$, $2^3\cdot 3^2$, $2^5\cdot 3$, $2^4\cdot 3^2$, $2^6\cdot 3$,  $2^5\cdot 3^2$, $2^7\cdot 3$, $2^6\cdot 3^2$, $2^7\cdot 3^2$, $2^8\cdot 3^2$ or $2^9\cdot 3^2$.

Let $\Gamma$ be non-normal and of valency $7$.  Pan {\em et al.}~\cite{PYL19} proved that if $\Aut(\Gamma)$ has solvable vertex stabilizers, then $G=\A_{n-1}$ and $\Gamma$ is $\A_n$-arc-transitive with $n = 7, 21, 63$ or $84$, and this was generalized to any prime $p\geq 11$ by Yin {\em et al.}~\cite{YFZC2021} that $\Aut(\Gamma)$ contains an arc-transitive non-abelian simple subgroup $T$ such that $G<T$ and $(G, T, p)=(\A_5, \PSL(2,11), 11)$, $(\A_5, \PSL(2,29), 29)$, $(\M_{22}, \M_{23}, 23)$, or $(\A_{n-1}, \A_n, p)$ with $n=pk\ell$, $k\ |\ \ell$ and $\ell\ |\ (p-1)$, where $k$ and $\ell$ have the same parity.
Recently, with no restriction on stabilizer for valency $7$, Li {\em et al.}~\cite{LMZ2022} characterised the group pairs  $(\soc(\Aut(\Gamma)/K),GK/K)$, where $K$ is a maximal intransitive normal subgroup of $\Aut(\Gamma)$.
Their result is as follows:

\begin{theorem}[{\cite[Theorem~1.1]{LMZ2022}}]\label{th:LMZ}
Let $G$ be a finite non-abelian simple group and let $\Gamma = \mathrm{Cay}(G,S)$ be a connected symmetric Cayley graph of valency $7$ on $G$ with $\alpha \in V(\Gamma)$. Then either $\Gamma$ is a  normal Cayley graph or one of the following holds:
\begin{enumerate}[\rm (1)]
\item $\mathrm{Aut}(\Gamma)$ acts quasiprimitively on $V(\Gamma)$ with $L = \mathrm{soc}(\mathrm{Aut}(\Gamma))$ being a non-abelian  simple group, and  $(L,G)=(\mathrm{P\Omega}^{+}(12,2),\mathrm{\Omega}(11,2))$, $(\mathrm{Sp}(6,4),\mathrm{PSU}(4,4))$ or $(\mathrm{A}_7,\mathrm{A}_6)$.
\item There exists a maximal intransitive normal subgroup $K$ of $A = \mathrm{Aut}(\Gamma)$ such that  the socle of $A/K$, say $\overline{L}$, is a simple group containing $\overline{G} = GK/K \cong G$ properly.
 Moreover, one of the following holds.
 \begin{enumerate}[\rm (i)]
 \item $\overline{L}$ is classical simple group, and $(\mathrm{Aut}(\Gamma),G)$ is one of the seven cases in Table~\ref{tb:LMZ}.
 \item Otherwise, up to isomorphism, we have either $(\mathrm{Aut}(\Gamma),G)\cong (\mathbb{Z}_7 \times  \mathrm{M}_{24}, \mathrm{M}_{23}) $, or $(\overline{L},\overline{G})=(\mathrm{A}_n,\mathrm{A}_{n-1})$ with $\overline{L}_{\overline{\alpha}}$ has a subgroup of index $n$ for $\overline{\alpha} \in V(\Gamma_K)$.
 \end{enumerate}
\end{enumerate}
\end{theorem}

\begin{table}[h]
\begin{center}
\caption{Candidates for $(\mathrm{Aut}(\Gamma) ,G)$ when $\overline{L}$ is classical.}\label{tb:LMZ}
\begin{tabular}{l|l|l}
\hline
& $\mathrm{Aut}(\Gamma)$ & $G$\\ \hline
1&  $[2^{18}].\mathrm{P\Omega}^{+}(12,2)$& $\mathrm{\Omega} (11,2)$\\
2& $(\mathbb{Z}_7\rtimes \mathbb{Z}_3) \times \mathrm{A}_8 $ & $\mathrm{PSL} (3,2)$ \\
3 &$((\mathbb{Z}_7\rtimes \mathbb{Z}_3) \times \mathrm{A}_8)\rtimes \mathbb{Z}_2 $ & $\mathrm{PSL}(3,2)$ \\
4& $\mathbb{Z}_2.\mathrm{PSU} (6,2)$& $ \mathrm{PSU}(5,2)$ \\
5&$\mathbb{Z}_2^2.\mathrm{PSU}(6,2)$& $ \mathrm{PSU}(5,2)$ \\
 &$\mathbb{Z}_4.\mathrm{PSU}(6,2)$& \\
6&$\mathbb{Z}_2^4.\mathrm{Sp}(6,4)$ &$ \mathrm{PSU}_4(4)$ \\
7&$(\mathbb{Z}_7\rtimes \mathbb{Z}_3) \times \mathrm{AGL}(3,2) $ & $\mathrm{PSL}(3,2)$\\
\hline
\end{tabular}
\end{center}
\end{table}

The motivation of this paper is to enhance the above result. We will prove that Theorem~\ref{th:LMZ}(1) except $(L,G) = (\mathrm{A}_7, \mathrm{A}_6)$, and Theorem~\ref{th:LMZ}(2)(i) cannot occur. Furthermore, only certain cases of Theorem~\ref{th:LMZ}(2)(ii) are possible, as outlined below in our Theorem~\ref{th:main-vald}:

\begin{theorem}\label{th:main-vald}
Let $\Gamma$ be a connected symmetric $7$-valent Cayley graph on a non-abelian simple group $G$. Then $\Gamma$ is normal or $\Aut(\Gamma)$ contains an arc-transitive non-abelian simple normal subgroup $T$ such that $R(G)<T$ and $(T,R(G))=(\A_{n},\A_{n-1})$ with $n=7$, $3\cdot 7$, $3^2\cdot 7$, $2^2\cdot 3\cdot 7$, $2^3\cdot3\cdot7$, $2^3\cdot3^2\cdot5\cdot7$, $2^4\cdot3^2\cdot5\cdot7$, $2^6\cdot3\cdot7$, $2^7\cdot3\cdot7$,
 $2^6\cdot3^2\cdot7$, $2^6\cdot3^4\cdot5^2\cdot7$, $2^7\cdot3^4\cdot5^2\cdot7$, $2^8\cdot3^4\cdot5^2\cdot7$, $2^{10}\cdot3^2\cdot7$, $2^{24}\cdot3^2\cdot7$. Moreover, $\soc(\Aut(\Gamma)/R)=(T\times R)/R$, where $R$ is the largest solvable normal subgroup of $\Aut(\Gamma)$.
\end{theorem}

Theorem~\ref{th:main-vald} is formulated with respect to the largest solvable normal subgroup $R$ of $\Aut(\Gamma)$, rather than a maximal intransitive normal subgroup $K$ of $\Aut(\Gamma)$ as given in Theorem~\ref{th:LMZ}. However, it follows from Lemma~\ref{lm:condi3} and Theorem~\ref{th:main-vald} that $R$ is indeed the maximal intransitive normal subgroup $K$. We observed a gap in Line 2 on page 1115 of \cite{LMZ2022}, which asserts that ``$\overline{L}$ is not regular on $V(\Gamma_K)$ and $\overline{L} \neq \overline{G}$ is symmetric on $\Gamma_K$ since the valency of $\Gamma_K$ is $7$". However, \cite{LMZ2022} did not provide a proof for the assertion $\overline{L} \neq \overline{G}$.
This gap is fixed in this paper: according to Lemma~\ref{lm:condi}(2), if $\overline{L}= \overline{G}$, then $\Gamma$ must be a normal Cayley graph on $G$.

%


In the literature, there are some other results on  non-normal arc-transitive Cayley graphs  on a non-abelian simple group with more restrictions, and related results can be found in \cite{DCF,DF2,FPW2002,Fang4,FangLW,Lu24,LiWZ24,LM2021,PXY22,PWZ2022} and the references therein.

\section{Preliminaries}

For a positive integer $n$, we denote by $\mz_n$, $D_n$, $F_n$, $\A_n$ and $\S_n$ the cyclic of order $n$, dihedral group of order $n$, Frobenius group of order $n$, alternating group of degree $n$, and symmetric group of degree $n$, respectively.
For two groups $N$ and $H$, denote by $N\times H$ the direct product of $N$ and $H$, and by $N.H$ an extension of $N$ by $H$. If this extension is split, it is denoted by $N\rtimes H$.

The following proposition is well-known in permutation group theory (see \cite[P.45:Exercises 2.5.6]{DM-book}) or \cite[Exercise 4.5 and 4.5']{Wielandt-book}),  which
has achieved a sort of folklore status.

\begin{proposition} \label{pro:tran-centra} Let $G$ be a transitive permutation group on a set $\Omega$. Then the centralizer $C_{\S_{\Omega}}(G)$ of $G$ in the symmetric group $\S_{\Omega}$ is semiregular. If $G$ is further regular, then $C_{\S_{\Omega}}(G)$ is also regular and isomorphic to $G$.
\end{proposition}

\noindent{\bf Remark:} A regular permutation group $G$ on a set $\Omega$ is permutation isomorphic to the right regular representation $R(G)$ of $G$, where $R(G)=\{R(g)\ |\ g\in G\}$ with $R(g): x\mapsto xg$ for every $x\in G$, is a regular permutation group in the symmetric group $\S_G$, and the centralizer $C_{\S_G}(R(G))$ of $R(G)$ in $\S_G$ is the left regular representation $L(G)$ of $G$, where $L(G)=\{L(g)\ |\ g\in G\}$ with $L(g): x\mapsto g^{-1}x$ for every $x\in G$. Furthermore, $R(G)\cong L(G)\cong G$.

\medskip
The following result follows from the classification of four-factor simple groups in $\cite{Huppert}$.

\begin{proposition}[{\cite[Theorem III]{Huppert}}]\label{pro:four-factor}
Let $G$ be a non-abelian simple group with $|G|\bigm|2^{24}\cdot3^4\cdot5^2\cdot7$. Then $G$ is one of the group listed in Table~$\ref{tb:2357simplegroup}$, where $\Pi(G)$ is the set of prime factors of $|G|$.

\begin{table}[htp!]
\caption{The non-abelian simple group $G$ with $|G|\bigm|2^{24}\cdot3^4\cdot5^2\cdot7$}\label{tb:2357simplegroup}%
\begin{center}
\begin{tabular}{ l l l l }
\hline
 $\Pi(G)$  &$\{2,3,5,7\}$      & $\{2,3,5\}$ & $\{2,3,7\}$ \\
\hline
$G$& $\J_2, \A_7, \A_8$ & $\A_5$ & $\PSL(2,7)$   \\
 &  $\A_9, \A_{10}, \PSL(3,4)$& $\A_6$ & $\PSL(2,8)$ \\
 &$\PSp(6,2)\cong \PGO(7,2)$ &$\PSU(4,2)$ & $\PSU(3,3)$\\
\hline
$|G|$ &$2^7\cdot3^3\cdot5^2\cdot7$, $2^3\cdot3^2\cdot5\cdot7$, $2^6\cdot3^2\cdot5\cdot7$  & $2^2\cdot3\cdot5$& $2^3\cdot3\cdot7$\\
&$2^6\cdot3^4\cdot5\cdot7$, $2^7\cdot3^4\cdot5^2\cdot7$, $2^6\cdot3^2\cdot5\cdot7$ &$2^3\cdot3^2\cdot5$& $2^3\cdot3^2\cdot7$\\
&$2^9\cdot3^4\cdot5\cdot7$ &$2^6\cdot3^4\cdot5$ & $2^{5}\cdot3^3\cdot7$\\
\hline
\end{tabular}
\end{center}
\end{table}
\end{proposition}

Let $G$ be a finite group. An extension $G=N.T$ of $N$ by $T$ is called a central extension if $N\leq Z(G)$, and $G$ is called a {\em covering group} of $T$ if $G=G'$ and $G=N.T$ is a central extension.
Every non-abelian simple group $T$ has a unique maximal covering group such that every covering group of $T$ is a factor group of this maximal covering group, see \cite[Chapter 5: Section 23]{Huppert-book}.
The center of this maximal covering group is called the \emph{Schur  multiplier} of $T$, denote by Mult($T$).

By \cite[Theorem~5.1.4]{K-Lie}, Mult$(\A_n)=\mz_2$ for $n \geq 5$ with $n \neq 6,7$, and Mult$(\A_n)=\mz_6$ for $n=6$ or 7. Then for every $n\geq 5$, $\A_n$ has a unique covering group of order $2|A_n|$, denoted by $2.\A_n$, and $\A_7$ has a unique covering group of order $3|\A_7|$ and  $6|\A_7|$, denoted by $3.\A_7$ and $6.\A_7$, respectively.

\begin{proposition}[{\cite[Proposition~2.6]{DFZh}}]\label{pro:shurmult}
For $n \geq 7$, all subgroups of index $n$ of $2.\A_n$ are isomorphic to $2.\A_{n-1}$.
\end{proposition}

An expression $G=HK$ of a group $G$ as the product of subgroups $H$ and $K$ is called a \emph{factorization} of $G$, where $H$ and $K$ are called factors.
Below is the well known Frattini argument.

\begin{proposition}\label{Frattini}
Let $M$ be a permutation group on a set $\Omega$, and  $G$   a subgroup of $M$.
Then  $G$ is transitive on $\Omega$ if and only if  $M=GM_v$ for some $v\in \Omega$.
\end{proposition}

The following result can be proved by using   the Frattini argument.

\begin{corollary}\label{coro:facs}
If $G=HK$ is a factorization of a group $G$, then $G=H^aK^b$ for any $a,b\in G$.
\end{corollary}

The next proposition is about the classification of non-abelian simple group factorizing into a non-abelian simple group and a vertex-stabilizer in Proposition~$\ref{pro:vertex-stab}$,  see \cite[Lemmas~3.1-3.8]{LMZ2022}.

\begin{proposition}[{\cite[Tables 2-4]{LMZ2022}}] \label{pro:factorization}
Let $M$ be a permutation group on a set $\Omega$, and $G$ a transitive subgroup of $M$. Assume that $G$ and $M$ are non-abelian simple subgroups, and $M_v$ is one of the group given in Proposition~$\ref{pro:vertex-stab}$ for $v\in \Omega$. Then $M$, $G$ and $M_v$ are listed in Table~\ref{tb:factorization 2}.
\end{proposition}
\begin{table}[h]
\caption{Candidates for $(M,G,M_v)$}\label{tb:factorization 2}%
\begin{center}
\begin{tabular}{@{}lllll@{}}
\hline Row & $M$ & $G$ & $M_v$ & remark \\
\hline 1 & $\PSp(8,2)$ & $\Omega^{-}(8,2)$ & $\S_7$ \\
2 & $\PSp(12,2)$ & $\Omega^{-}(12,2)$ & $\A_7 \times \A_6, \S_7 \times \S_6$, $(\A_7 \times \A_6) \rtimes \mz_2$ \\
3 & $\P\Omega^+(8,2)$ & $\Omega(7,2)$ & $\A_7, \S_7$ \\
4 & $\P\Omega^+(12,2)$ & $\Omega(11,2)$ &$\S_7 \times \S_6, \mz_2^6 \rtimes(\SL(2,2) \times \SL(3,2))$, \\
 & & & {$[2^{20}] \rtimes(\SL(2,2) \times \SL(3,2))$}, \\
 & & & $\SL(3,2) \times \S_4, \A_7 \times \A_6,(\A_7 \times \A_6) \rtimes \mz_2$\\
5& $\PSL(6,2)$ & $\PSL(5,2)$ & $\SL(3,2) \times \S_4, \mz_2^6 \rtimes(\SL(2,2) \times \SL(3,2))$ \\
6& $\Sp(6,2)$&$\A_8$&$\A_7, \S_7$\\
7 & $\PSL(4,2)$ & $\A_6$ & $\SL(3,2), \A_7$ \\
8 & $\PSL(4,2)$ & $\PSL(3,2)$ & $\A_7$ \\
9 & $\PSp(8,2)$ & $\PSp(4,4)$ & $\mz_2^6 \rtimes(\SL(2,2) \times \SL(3,2))$ \\
10 & $\PSp(6,2)$ & $\PSU(4,2)$ & $\SL(3,2), \A_7, \S_7$, $\mz_2^3 \rtimes \SL(3,2)$ \\
11 & $\PSU(6,2)$ & $\PSU(5,2)$ & $\mz_2^3 \rtimes \SL(3,2)$ \\
12 & $\Sp(6,4)$ & $\PSU(4,4)$ & $\SL(3,2) \times \S_4$ \\
13 & $\PSU(4,3)$& $\PSp(4,3)$&$\A_7$\\
14 & $\A_n$ & $\A_{n-1}$ & $n\bigm||M_v|$ & $n \geq 6$ \\
15 & $\A_7$ & $\A_5$ & $\PSL(3,2)$ & \\
16 & $\A_n$ & $\M_n$ & $\A_7, \S_7$ & $n \in\{11,12\}$\\
17 & $\A_9$ & $\PSL(2,8)$ & $\A_7, \S_7$ &  \\
18 & $\A_8$ & $\A_k$ & $\AGL(3,2)$ & $k \in\{5,6,7\}$ \\
& & & $\SL(3,2)$ &$k \in\{6,7\}$ \\
19 & $\M_{24}$ & $\M_{23}$ & $\mz_2^6 \rtimes(\SL(3,2) \times \S_3)$, $\SL(3,2)$&\\
\hline
\end{tabular}
\end{center}
\end{table}

\medskip
Let $G$ be a finite group and $p$ a prime. A {\em $p$-modular representation} of $G$ is a homomorphism from $G$ to $\GL(n,\F)$ for some $n$ and some field $\F$ of characteristic $p$. Similarly, a {\em projective $p$-modular representation} of $G$ is a homomorphism from $G$ to $\PGL(n,\F)$. Obviously, $G$ can embed in $\PGL(n,\F)$ for some $n$ and $\F$ if and only if $G$ has a faithful projective $p$-modular representation of degree $n$. The smallest values of $n$ for that $G$ can embed in $\PGL(n,\F)$ can be found in~\cite{K-Lie}. Write  $R_p(G)=\min\{n \mid G\lesssim \PGL(n,\F), F \mbox{ is a field of characteristic }p \}$ and $R(G)=\min\{R_p(G)\mid \text{all primes}\ p\}$.

\begin{proposition}[{\cite[Theorem 5.3.7]{K-Lie}}]\label{pro:lessGL}
For $n \geq 5$, we have $R(\A_n) \geq n-4$. More specifically,
\begin{itemize}
\item [$(i)$] for $n \geq 9$, we have $R(\A_n)=n-2$;
\item [$(ii)$] for $5 \leq n \leq 8$, the values of $R_p(\A_n)$ are listed in Table~\ref{tb:RAn}:
\begin{table}[htp!]
\caption{The classification of $R_p(\A_n)$}\label{tb:RAn}%
\begin{center}
\begin{tabular}{ c c c c c }
\hline$n$ & $R_2(\A_n)$ & $R_3(\A_n)$ & $R_5(\A_n)$ & $R_p(\A_n), p \geq 7$ \\
\hline 5 & 2 & 2 & 2 & 2 \\
6 & 3 & 2 & 3 & 3 \\
7 & 4 & 4 & 3 & 4 \\
8 & 4 & 7 & 7 & 7 \\
\hline
\end{tabular}
\end{center}
\end{table}
\end{itemize}
\end{proposition}

\medskip

Let $\Gamma$ be a graph. An {\em $s$-arc} of $\Gamma$ is an ordered $(s+1)$-tuple $(v_0,v_1,...,v_s)$ of vertices of $\Gamma$ such that $v_{i-1}$ is adjacent to $v_i$ for $1\leq i\leq s$, and $v_{i-1}\neq v_{i+1}$ for $1\leq i<s$. The graph $\Gamma$, with $G\leq \Aut(\Gamma)$, is said to be {\em $(G,s)$-arc-transitive} if $G$ is transitive on the $s$-arc set of $\Gamma$, and  {\em $(G, s)$-transitive} if it is $(G, s)$-arc-transitive but not $(G, s+1)$-arc-transitive. The following results are about the structure of vertex-stabilizer of a $G$-arc-transitive graph of valency 7.

\begin{proposition}[{\cite[Theorem~1.1]{GuoLiHua}}]\label{pro:vertex-stab}
Let $\Gamma$ be a connected $(G, s)$-transitive graph of valency $7$ with $s \geq 1$ and $G \leq \Aut(\Gamma)$.
Then $s \leq 3$ and one of the following statements holds for $v \in V(\Gamma)$:
\begin{enumerate}[\rm (1)]
\item $s=1$ and $G_v \cong \mz_7$, $D_{14}, F_{21}$, $D_{28}$ or $F_{21} \times \mz_3$. In particular, $|G_v|=7$, $2\cdot7$, $3\cdot7$, $2^2\cdot7$ or $3^2\cdot7$;
\item $s=2$ and $G_v \cong F_{42}$, $F_{42} \times \mz_2$, $F_{42} \times \mz_3$, $\PSL(3,2)$, $\A_7$, $\S_7$, $\mz_2^3 \rtimes \SL(3,2)$ or $\mz_2^4 \rtimes \SL(3,2)$. In particular, $|G_v|=2\cdot3\cdot7$, $2^2\cdot3\cdot7$, $2\cdot3^2\cdot7$, $2^3\cdot3\cdot7$, $2^3\cdot3^2\cdot5\cdot7$, $2^4\cdot3^2\cdot5\cdot7$, $2^6\cdot3\cdot7$ or $2^7\cdot3\cdot7$;
\item $s=3$ and $G_v \cong F_{42} \times \mz_6$, $\PSL(3,2) \times \S_4$, $\A_7 \times \A_6$, $\S_7 \times \S_6$, $(\A_7 \times \A_6) \rtimes \mz_2$, $\mz_2^6 \rtimes(\SL(2,2) \times \SL(3,2))$ or $[2^{20}] \rtimes(\SL(2,2) \times \SL(3,2))$. In particular, $|G_v|=2^2\cdot3^2\cdot7$, $2^6\cdot3^2\cdot7$, $2^6\cdot3^4\cdot5^2\cdot7$, $2^8\cdot3^4\cdot5^2\cdot7$, $2^7\cdot3^4\cdot5^2\cdot7$, $2^{10}\cdot3^2\cdot7$ or $2^{24}\cdot3^2\cdot7$.
\end{enumerate}
\end{proposition}

A subgroup $H$ of a group $G$ is said to be core-free in $G$ if $\bigcap_{g \in G} H^g=1$, that is, $H$ contains no nontrivial normal subgroup of $G$.
Let $G$ be a group, $H$ a core-free subgroup of $G$  and $D$ some double cosets of $H$ in $G$ such that $D=D^{-1}$.
The \emph{coset graph} $\mathrm{Cos}(G, H, HDH)$ of $G$ with respect to $H$ and $D$ is defined to has vertex set $[G: H]$, the set of right cosets of $H$ in $G$, and $Hx$ is adjacent to $H y$ with $x, y\in G$ if and only if $y x^{-1} \in HDH$. The following assertion is due to Sabidussi~\cite{Sab64}.

\begin{proposition}\label{lm:Cosarc}
Let $\Gamma$ be a connected $G$-arc-transitive graph of valency $d$ with $G
\leq \Aut(\Gamma)$ and $v \in V(\Gamma)$. Then $\Gamma \cong$ $\mathrm{Cos}(G, G_v, G_vgG_v)$ for a $2$-element $g$ satisfying:
$$
g \in N_{G}(G_v \cap G_v^{g}), g^{2} \in G_v,\langle G_v, g\rangle=G \mbox{ and } |G_v: G_v \cap G_v^{g}|=d,
$$
where $\{v,v^g\}\in E(\Gamma)$. Conversely, if $H$ is a core-free subgroup of a group $X$ and $x \in X$ satisfying the above conditions by replacing $g$, $G$ and $G_v$ by $x$, $X$ and $H$ respectively, then $\mathrm{Cos}(X, H, HxH)$ is a connected $X$-arc-transitive graph of valency $d$.
\end{proposition}

A $2$-element $g$ is called to be \emph{feasible} to the pair $(G, G_v)$, if $g$, $G$ and $G_v$ satisfy the conditions in Proposition~\ref{lm:Cosarc}.

\medskip
Let $\Gamma$ be a $G$-vertex-transitive graph and let $N$ be an intransitive normal subgroup of $G$. The normal quotient $\Gamma_N$ of $\Gamma$ with respect to $N$ is defined to be the graph such that $V(\Gamma_N)$ is the set of orbits $N$ on $V(\Gamma)$ and two orbits $B, C$ of $N$ are adjacent if and only if $\{b,c\}\in E(\Gamma)$ for some $b\in B$ and $c\in C$. If $\Gamma$ and $\Gamma_N$ have the same valency, $\Gamma$ is called an \emph{normal $N$-cover} of $\Gamma_N$.

\begin{proposition}[{\cite{Lorimer}}]\label{pro:quotientgraph}
Let $\Gamma$ be a connected $X$-arc-transitive graph of prime valency, and let $N\unlhd X$ has at least three orbits on $V(\Gamma)$.
Then
\begin{enumerate}[\rm (a)]
\item $N$ is the kernel of $X$ on $V(\Gamma_N)$ and semi-regular on $V(\Gamma)$, $X/N \leq \Aut(\Gamma_N)$, $\Gamma_N$ is a connected $X/N$-arc-transitive graph, and $\Gamma$ is a normal $N$-cover of $\Gamma_N$.
\item $X_v \cong (X/N)_{\alpha}$ for any $v \in V(\Gamma)$ and $\alpha \in V(\Gamma_N)$.
\end{enumerate}
\end{proposition}


\section{Proof of Theorem~\ref{th:main-vald}}

First we prove a simple lemma which will be used frequently.

\begin{lemma}\label{lem:lem1} Let $G$ be a simple group with $G\leq T_1\times T_2\times\cdots\times T_k$, where $k\geq 1$ and $T_i$ is a simple group for every $1\leq i\leq k$. Then $G$ is isomorphic to a subgroup of some $T_i$.
\end{lemma}

\begin{proof} The lemma is clearly true for $G\leq T_1$. Thus, we assume that $G\not\leq T_1$. Since $G\leq T_1\times T_2\times\cdots\times T_k$, there is some $2\leq i\leq k$ such that $G\leq T_1\times T_2\times\cdots\times T_i$ and $G\not\leq T_1\times T_2\times\cdots\times T_{i-1}$. Since $G$ is simple, $G\cap (T_1\times T_2\times\cdots\times T_{i-1})=1$, and hence $G=G/(G\cap (T_1\times T_2\times\cdots\times T_{i-1}))\cong G(T_1\times T_2\times\cdots\times T_{i-1})/(T_1\times T_2\times\cdots\times T_{i-1})\leq G(T_1\times T_2\times\cdots\times T_i)/(T_1\times T_2\times\cdots\times T_{i-1})=(T_1\times T_2\times\cdots\times T_i)/(T_1\times T_2\times\cdots\times T_{i-1})\cong T_i$. This completes the proof. \end{proof}

To prove Theorem~\ref{th:main-vald}, in what follows we always make the following assumption:

\medskip
\noindent{\bf Assumption:} $\Gamma$ is a connected $7$-valent $X$-arc-transitive graph with $X\leq A=\Aut(\Gamma)$, and $G\leq  X$ is a transitive nonabelian simple  group on $V(\Gamma)$.

\medskip
Let us begin by proving a series of lemmas which will be used in the proof of Theorem~\ref{th:main-vald}.

\begin{lemma}\label{NonabeN} Under Assumption, let $K$ be a non-abelian simple normal subgroup of $X$ with $G\cap K=1$. Then $KG=K\times G$, and if $G$ is regular, then $K$ is isomorphic to a proper subgroup of $G$.
\end{lemma}

\begin{proof}
Write $B=KG$ and $C=C_{B}(K)$, the centralizer of $K$ in $B$. Then $B\leq X$, $K\unlhd B$, $C\unlhd B$ and $C\cap G\unlhd G$. Since $K$ is non-abelian simple, we have $C\cap K=1$ and hence $KC=K\times C$. Since $K\cap G=1$, we have $G\cong B/K$ and so $C\cong  KC/K\unlhd B/K \cong G$. Since $G$ is simple, we have $C=1$ or $C\cong G$.

Suppose $C=1$. By the $N/C$-theorem,  $B\cong B/C \lesssim \Aut(K)$, and so $G\cong B/K\lesssim \Aut(K)/\Inn(K)$, where $\Inn(K)$ is the inner automorphism group of $K$, that is, the group of automorphisms of $K$ induced by conjugate of elements of $K$. However, by the Schreier's Conjecture $\Aut(K)/\Inn(K)$ is solvable, forcing that $G$ is solvable, a contradiction.

Thus, $C\cong G$. Now we have $|KC|=|K||C|=|K||G|$. Since $G\cap K=1$, we have $|B|=|K||G|$, and since $KC\leq B$, we have $B=KC=K\times C$. Again by the simplicity of $G$, we have $C\cap G=1$ or $C\cap G=G$.

Suppose $C\cap G=1$. Then $G\cong GC/C\leq KC/C\cong K$, that is, $G$ is isomorphic to a subgroup of $K$, implying  $|G|\mid |K|$. If $K$ has two orbits on $V(\Gamma)$, by the normality of $K$ in $X$ and arc-transitivity of $X$, $\Gamma$ is bipartite. Then the transitivity of $G$ on $V(\Gamma)$ implies that $G$ has a normal subgroup of index $2$, contradicting the simplicity of $G$. Thus, $K$ is transitive on $V(\Gamma)$ or has at least three orbits.
If $K$ has at least three orbits, Proposition~\ref{pro:quotientgraph} implies that $K$ is semiregular on $V(\Gamma)$, implying $|K|$ is proper divisor of $|G|$, contradicting $|G|\mid |K|$. Thus, $K$ is transitive on $V(\Gamma)$.

Recall $C=\mathrm{C}_{B}(K)$. By Proposition~\ref{pro:tran-centra},  $C$ is semiregular, and since $C\cong G$, $C$ is regular on $V(\Gamma)$, which returns that $K$ is also regular on $V(\Gamma)$. Thus, we may write $\Gamma=\Cay(K,S)$ with $R(K)\times L(K)\leq A$, where $R(K)$ and $L(K)$ are the right and left regular representations of $K$ respectively. It is easy to see that $R(K) \times L(K)=R(K):\Inn(K)$, and so $\Inn(K) \leq \Aut(K, S)$, implying that $S$ is a union of some conjugate class of elements of $K$. Wrtie $S=x_1^K \cup \cdots \cup x_m^K$ for some $x_1,\ldots,x_m\in K$. Since $\Gamma$ is a $7$-valent connected graph, we have $|S|=7$ and $\langle S\rangle=K$, and since $\Inn(K) \leq \Aut(K, S)$, $K$ has a faithful transitive permutation representation of degree at most $7$ and hence $K\lesssim \A_7$. By Atlas~\cite{Atlas}, $K\in \{\A_5,\A_6,\A_7,\PSL(3,2)\}$. 
However, by checking the lengths of conjugate classes  of these four simple groups in the Atlas~\cite{Atlas}, we see that none of them satisfies $S=x_1^K \cup \cdots \cup x_m^K$ and $\vert S \vert=7$, a contradiction.

The above contradiction implies $C\cap G=G$. It follows that $G\leq C$, and since $C\cong G$, we have $C=G$ and hence $B=K\times C=K\times G$.

Let $G$ be regular. By the remark of Proposition~\ref{pro:tran-centra}, $K$ is isomorphic to a subgroup of $G$. A similar  argument to the above paragraph implies that $K\not\cong G$, and hence $K$ is isomorphic to a proper subgroup of $G$. This completes the proof.
\end{proof}

For a group $G$, the \emph{radical} of $G$, denote by $\mathrm{rad}(G)$, is its largest soluble normal subgroup, and the \emph{socle} of $G$, denoted by soc$(G)$, is the product of all minimal normal subgroups of $G$. A group $G$ is said to be \emph{almost simple} if soc$(G)$ is a nonabelian simple group.

\begin{lemma}\label{lm:trival radical}
Under Assumption, assume $X$ has trivial radical and $G$ is nonnormal in $X$. Then $X$ is almost simple with $G<\mathrm{soc}(X)$, $\Gamma$ is $\mathrm{soc}(X)$-arc-transitive, and $(\mathrm{soc}(X),G)=(\A_{8}, \PSL(3,2))$, $(\M_{24}, \M_{23})$, or $(\A_{n},\A_{n-1})$ with $n\geq 7$ and $n\bigm|2^{24}\cdot3^4\cdot5^2\cdot7$. If $G$ is further regular, then
$(\mathrm{soc}(X),G)$=$(\A_{n},\A_{n-1})$ with
$n=7$, $3\cdot 7$, $3^2\cdot 7$, $2^2\cdot 3\cdot 7$, $2^3\cdot3\cdot7$, $2^3\cdot3^2\cdot5\cdot7$, $2^4\cdot3^2\cdot5\cdot7$, $2^6\cdot3\cdot7$, $2^7\cdot3\cdot7$,
 $2^6\cdot3^2\cdot7$, $2^6\cdot3^4\cdot5^2\cdot7$, $2^8\cdot3^4\cdot5^2\cdot7$, $2^7\cdot3^4\cdot5^2\cdot7$, $2^{10}\cdot3^2\cdot7$, $2^{24}\cdot3^2\cdot7$.
\end{lemma}

\begin{proof}
At first, we prove that $X$ is an almost simple.
Let $N$ be a minimal normal subgroup of $X$.
Since $X$ has trivial radical, we may assume $N=T^d$ for a positive integer $d$ and a non-abelian simple group $T$.
Let $K=N G \leq X$. Since $N\cap G\unlhd G$, the simplicity of $G$ implies that $N\cap G=1$ or $N\cap G=G$.

\medskip
\noindent{\bf Claim:} If $N\cap G=1$ then $K=N\times G$.

Let $N\cap G=1$. Then $|K|=|N||G|$. Since $\Gamma$ is $G$-vertex-transitive, Proposition~\ref{Frattini} implies that $K=K_{v}G$ for $v\in V(\Gamma)$, and $|K|=|G||K_v|/|G_v|$. Thus, $|N|=|K_v|/|G_v|$, and by Proposition~\ref{pro:vertex-stab}, $|N|\ |\ 2^{24}\cdot3^4\cdot5^2\cdot7$. Since $N=T^d$, the simplicity of $T$ implies that $|T|$ has at least three distinct prime factors. It follows that $d\leq 2$ and if $d=2$ then $|N|$ is a $\{2,3,5\}$-group.

By Lemma~\ref{NonabeN}, the claim holds for $d=1$. Thus, we may assume $d=2$. Then $N=T^2$ with $T$ a $\{2,3,5\}$-group.

Suppose that $G$ has a nontrivial conjugate action on $N$. By the simplicity of $G$, we may have
\[
 G \leq \Aut(N) =(T^2.\Out(T)^2)\rtimes \S_2.
 \]
Since $\Aut(N)/T^2$ is solvable, $G\cap T^2\not=1$, implying $G\leq T^2$. By Lemma~\ref{lem:lem1}, $G$ is isomorphic to a subgroup of $T$, that is, $G\cong L\lesssim T$.

Let $N$ have two orbits on $V(\Gamma)$. Since $X$ is arc-transitive and $N\unlhd X$, $\Gamma$ is bipartite, forcing that $G$ admits an index $2$ normal subgroup, contradicting the simplicity of $G$. Let $N$ have at least three orbits on $V(\Gamma)$. By Proposition~\ref{pro:quotientgraph}, $N$ is semiregular on $V(\Gamma)$, implying $|N|\ |\ |V(\Gamma)|$ and so $|N|\ |\ |G|$, contradicting that $G\lesssim T$.
Thus, $N$ is transitive on $V(\Gamma)$. If $N$ is regular, then $|T|< |N|=|V(\Gamma)|\leq |G|$, contradicting $G\lesssim T$. If $N$ is not regular, then $N_v^{\Gamma(v)}\not=1$ for some $v\in V(\Gamma)$, and since $X_v$ is primitive on $\Gamma(v)$ ($|\Gamma(v)|=7$) and $N_v\unlhd X_v$, $N_v$ is transitive on $\Gamma(v)$. Thus, $7\mid |N_v|$, which is impossible because $N$ is a $\{2,3,5\}$-group.

Thus, $G$ has a trivial conjugate action on $N$, that is, $K=N\times G$, as claimed.

\medskip Let $\soc(X)=N_1 \times \cdots \times N_m$, where $N_1,\ldots,N_m$ are all minimal normal subgroups of $X$.

Suppose that $G\cap N_i=1$ for every $i\in \{1,\ldots,m\}$.
By  claim, $G\leq C_X(\soc(X))$ and so $C_X(\soc(X))\not=1$, where $C_X(\soc(X))$ is the centralizer of $\soc(X)$ in $X$. Note that $C_X(\soc(X))\cap \soc(X)=Z(\soc(X))$, the center of $\soc(X)$.
Since $X$ has a trivial radical, $C_X(\soc(X))\cap \soc(X)=1$, and so a minimal normal subgroup of $X$ contained in $C_X(\soc(X))$ is not a subgroup of $\soc(X)$, which is clearly impossible because $\soc(X)$ contains all minimal normal subgroup of $X$. Thus, there exist some $i\in \{1,\ldots,m\}$ such that $G\cap N_i\neq 1$.

For notation simplicity, write $M=N_i$. Then $G\cap M\neq 1$, implying $G\leq M$. Since $G$ is nonnormal in $X$, $G$ is proper subgroup of $M$, and since $G$ is transitive, we have $M_v\not=1$ for $v\in V(\Gamma)$. By the arc-transitivity of $X$, we have $7\mid |M_v|$ and hence $M$ is arc-transitive. Write $M=T_1\times \cdots \times T_\ell$, where $T_1\cong\cdots \cong  T_\ell \cong T$ with $T$ a non-abelian simple group. By Lemma~\ref{lem:lem1}, we have $G\lesssim T$.

Suppose $\ell\geq 2$. Since $N$ is arc-transitive, $T_1$ cannot have at least three orbits; otherwise, $T_1$ is semiregular and hence $|T|=|T_1|<|V(\Gamma)|\leq |G|$, contradicting $G\lesssim T$. Also, $T_1$ cannot have two orbits, because otherwise $\Gamma$ is bipartite and hence $G$ has an index $2$ normal subgroup, contradicting the simplicity of $G$. Thus, $T_1$ is transitive, and so $|T|\geq |V(\Gamma)|$. By Proposition~\ref{pro:tran-centra},  $T_2\times \cdots \times T_\ell \leq C_N(T_1)$ is semiregular on $V(\Gamma)$, forcing that $|T_2\times \cdots \times T_\ell | \mid |V(\Gamma)|$. Since $|T_2\times \cdots \times T_d|=(d-1)|T|\geq |T|$, we have $d=2$ as $|T|\geq |V(\Gamma)|$. Furthermore, $T_1$ and $T_2$ are regular on $V(\Gamma)$, forcing that $|T|=|G|$ and $G$ is regular because $G\lesssim T$. Clearly, $G\cap T_1=1$ or $G\cap T_2=1$. By taking $X=N$ and $K=T_1$ or $T_2$, Lemma~\ref{NonabeN} implies that $T$ is isomorphic a proper subgroup of $G$, contradicting $|T|=|G|$.

Thus, $\ell=1$, that is, $M$ is a non-abelian simple group. Now we claim $M=\soc(X)$.

We argue by contradiction that $\soc(X)$ contains another minimal normal subgroup $L$ of $X$ such that $L\not=M$. Then $Y=M\times L$. Recall that $M$ is arc-transitive, and so $Y$ is arc-transitive. Since $Y=MY_v$ for $v\in V(\Gamma)$, we have $L\cong Y/M=LY_v/M\cong Y_v/M_v$, where $Y_v$ and $M_v$ is the group in Proposition~\ref{pro:vertex-stab}. Note that $L$ is a product of non-abelian simple groups. By Proposition~\ref{pro:vertex-stab}, $L\cong \A_6$, and either $(Y_v,M_v)=(\A_7\times \A_6,\A_7)$ or $((\A_7\times \A_6)\rtimes \mz_2, \S_7)$. Pick an element $a \in L$ of order $4$ and set $Z=M\times \langle a\rangle$. Then $M\unlhd Z$, $Z=MZ_v$, and $Z$ is arc-transitive. Furthermore, $Z_v/M_v\cong MZ_v/M=Z/M\cong \mz_4$, which is impossible by Proposition~\ref{pro:vertex-stab}.
Thus, $\soc(X)=M$, that is, $X$ is almost simple.

\medskip
Now $\Gamma$ is $M$-arc-transitive and $G$-vertex-transitive, and Proposition~\ref{Frattini} implies $M=GM_v$ for $v\in V(\Gamma)$.
By Proposition~\ref{pro:factorization}, $(M,G,M_v)$ lies in Table~\ref{tb:factorization 2}, and by Proposition~\ref{lm:Cosarc}, we may assume $\Gamma=\mathrm{Cos}(M,M_v,M_vgM_v)$ for some feasible element $g$ of $M$, that is, $g \in N_{M}(M_{uv})$, $g^2\in M_v$, $\langle M_v, g\rangle=M$ and $|M_v: M_{uv}|=7$, where $\{u,v\}$ is an edge of $\Gamma$.
For each possible $(M,G,M_v)$ in Table~\ref{tb:factorization 2}, we will determine whether exists such a graph $\Gamma$ by two steps: first verify whether $(M,G,M_v)$ indeed provides a factorization $M=GM_v$; then check for the existence of a feasible element $g$.
We use the following {\sc Magma}~\cite{Magma} code for the computation:
\begin{verbatim}
f:=function(M,orderG,orderMv)
Gs:=Subgroups(M:OrderEqual:=orderG);Mvs:=Subgroups(M:OrderEqual:=orderMv);
for i in [1..#Gs] do  G:=Gs[i]`subgroup;
for j in [1..#Mvs] do Mv:=Mvs[j]`subgroup;
if (#G*#Mv) eq (#(G meet Mv)*#M) then
print "find one factorization",GroupName(G),GroupName(Mv);
Mvus:=Subgroups(Mv:IndexEqual:=7);
for k in [1..#Mvus] do Mvu:=Mvus[k]`subgroup; NMvu:=Normalizer(M,Mvu);
if #sub<M|Mv,NMvu> eq #M  then
Tr:=Transversal(NMvu,Mvu);
gs:=[g: g in Tr|#(Mv meet Mv^g)*7 eq #Mv and #sub<M|Mv,g> eq #M and g^2 in Mv];
print "possible g",#gs; end if;end for;end if;end for;end for;return "";
end function;
\end{verbatim}
To verify whether $(M,G,M_v)$ provides  a factorization $M=GM_v$, we compute all conjugacy classes of subgroups of $M$ with orders $\vert G\vert$ and $\vert M_v \vert $.
According to  Corollary~\ref{coro:facs}, we may take $G$ and $M_v$ to be any group in the corresponding conjugacy classes.
Then factorization $M=GM_v$ holds if and only if $\vert M \vert \cdot \vert G \cap M_v\vert =\vert G\vert \cdot\vert M_v\vert$.
To check for the existence of a feasible element $g$, we examine the right transversal of $M_{vu}$ in $N_{M}(M_{uv})$, because if $g' \in M_{uv}g$ then $M_vg'M_v=M_vgM_v$ and hence $\mathrm{Cos}(M,M_v,M_vg'M_v)=\mathrm{Cos}(M,M_v,M_vgM_v)$.
For a specific triplet $(M,G,M_v)$, such as $ (\PSp(8,2),\Omega^-(8,2),\S_7)$ listed in  Row~1 of Table~\ref{tb:factorization 2}, we can use the above {\sc Magma}~\cite{Magma} function  with the following command:
\begin{verbatim}
f(PSp(8,2),#OmegaMinus(8,2),#Sym(7));
\end{verbatim}
It turns out that there is no such a factorization $\PSp(8,2)=\Omega^-(8,2)\S_7$ .
The computation results for all possible $(M,G,M_v)$ in Table~\ref{tb:factorization 2}, excluding Row 14, are listed in Table~\ref{tb:computation}. We note that, according to~\cite[Theorem~1.1]{LX2019}, $\PSp(12,2)$ and $\P\Omega^+(12,2)$ have no factorization with one factor having socle  $\mathrm{A}_6 \times \mathrm{A}_7$.

\begin{table}[h]
\caption{Computation results for $(M,G,M_v)$ in Table~\ref{tb:factorization 2}, excluding Row 14.}\label{tb:computation}%
\begin{center}
\begin{tabular}{ llllll }
\hline Row & $M$ & $G$ & $M_v$ & $M=GM_v$? & $g$ exists?\\
\hline 1 & $\PSp(8,2)$ & $\Omega^{-}(8,2)$ & $\S_7$ & No & \\
2 & $\PSp(12,2)$ & $\Omega^{-}(12,2)$ & $\A_7 \times \A_6$, $\S_7 \times \S_6 $, $(\A_7 \times \A_6) \rtimes \mz_2$&No & \\

3 & $\P\Omega^+(8,2)$ & $\Omega(7,2)$ & $\A_7$,  $\S_7$  &Yes & No \\
4 & $\P\Omega^+(12,2)$ & $\Omega(11,2)$ &$\SL(3,2) \times \S_4$, $  \mz_2^6 \rtimes(\SL(2,2) \times \SL(3,2))$  &No &   \\
&   & &   $[2^{20}] \rtimes(\SL(2,2) \times \SL(3,2))$  &No&   \\
&   & &   $\A_7 \times \A_6$, $\S_7 \times \S_6 $, $(\A_7 \times \A_6) \rtimes \mz_2$   &No&   \\

5& $\PSL(6,2)$ & $\PSL(5,2)$ & $\SL(3,2) \times \S_4$, $\mz_2^6 \rtimes(\SL(2,2) \times \SL(3,2))$ &No&   \\
6& $\Sp(6,2)$&$\A_8$&$\A_7$, $ \S_7$&No&  \\
7 & $\PSL(4,2)$ & $\A_6$ & $\A_7$&No&  \\
&   & &   $\SL(3,2) $  &Yes&No   \\
8 & $\PSL(4,2)$ & $\PSL(3,2)$ & $\A_7$ &Yes&Yes \\
9 & $\PSp(8,2)$ & $\PSp(4,4)$ & $\mz_2^6 \rtimes(\SL(2,2) \times \SL(3,2))$ &No& \\
10 & $\PSp(6,2)$ & $\PSU(4,2)$ & $\SL(3,2) $, $ \A_7 $, $ \S_7 $   &No&  \\
&   & &   $ \mz_2^3 \rtimes \SL(3,2)$  &Yes& No   \\
11 & $\PSU(6,2)$ & $\PSU(5,2)$ & $\mz_2^3 \rtimes \SL(3,2)$ &No&\\
12 & $\Sp(6,4)$ & $\PSU(4,4)$ & $\SL(3,2) \times \S_4$ &No& \\
13 & $\PSU(4,3)$& $\PSp(4,3)$&$\A_7$ &No&\\
15 & $\A_7$ & $\A_5$ & $\PSL(3,2)$ & Yes& No  \\
16 & $\A_{11}$ ($\A_{12}$)  & $\M_{11}$ ($\M_{12}$)  & $\A_7 $, $ \S_7$ & Yes& No  \\
17 & $\A_9$ & $\PSL(2,8)$ & $\A_7 $, $ \S_7$  & Yes& No  \\
18 & $\A_8$ & $\A_5,\A_6,\A_7$ & $\AGL(3,2)$ & Yes& No   \\
& & $ \A_6,\A_7$& $\SL(3,2)$ & Yes& No  \\
19 & $\M_{24}$ & $\M_{23}$ & $\mz_2^6 \rtimes(\SL(3,2) \times \S_3)$& Yes& No \\
&  &   &  $\SL(3,2)$& Yes& Yes  \\
\hline
\end{tabular}
\end{center}
\end{table}

In conclusion, the triplet $(M,G,M_v)$ can only appear in Row~$8$, Row~$14$, or Row~$19$ of Table~\ref{tb:factorization 2}, and $(\mathrm{soc}(X),G)=(M,G)=(\A_{8}, \PSL(3,2))$, $(\M_{24}, \M_{23})$ or $(\A_{n},\A_{n-1})$ with $n\geq 7$ and $n\bigm|2^{24}\cdot3^4\cdot5^2\cdot7$.

Assume further $G$ is regular.  Then $|M_v|=|M|/|G|$.
If $(M,G)=(\A_{8}, \PSL(3,2))$ or $(\M_{24},\M_{23})$, then by Table~\ref{tb:computation} $  M_v=\A_7$ or $\PSL_3(2)$, respectively, which contradicts $|M_v|=|M|/|G|$.
Therefore, $(M,G)=(\A_{n},\A_{n-1})$ with $n\geq 7$, and by Proposition~\ref{pro:vertex-stab} and~\cite{PYL19}, we have $n=7$, $3\cdot 7$, $3^2\cdot 7$, $2^2\cdot 3\cdot 7$, $2^3\cdot3\cdot7$, $2^3\cdot3^2\cdot5\cdot7$, $2^4\cdot3^2\cdot5\cdot7$, $2^6\cdot3\cdot7$, $2^7\cdot3\cdot7$,
 $2^6\cdot3^2\cdot7$, $2^6\cdot3^4\cdot5^2\cdot7$, $2^7\cdot3^4\cdot5^2\cdot7$, $2^8\cdot3^4\cdot5^2\cdot7$, $2^{10}\cdot3^2\cdot7$ or $2^{24}\cdot3^2\cdot7$.
\end{proof}

%

\begin{lemma}\label{lm:condi}
Under Assumption, let $G$ be regular on $V(\Gamma)$ and $R$ the radical of $\Aut(\Gamma)$. Then
\begin{enumerate}[\rm (1)]
\item  if $RG\neq R \times G$, then $G\lesssim \GL(d,r)$ for some prime $r$ and integer $d\geqslant 2$ with $r^{d}\bigm||R|$;
\item  if $7\bigm| |R|$, then $RG = R\times G$.
\end{enumerate}
\end{lemma}

\begin{proof}

Set $B=RG$. Since $R$ is solvable and $G$ is non-abelian simple, we have $R\cap G=1$ and $|B|=|R||G|$. Since $R$ is solvable, $B$ has a series of normal subgroup $R_i$ such that
\[
B>R=R_s>\cdots>R_1>R_0=1,
\]
where $R_i \unlhd B$ and $R_{i+1} / R_i$ is elementary abelian for every $0\leq i\leq s-1$.

To prove part~(1), assume $B\not=R\times G$. Then there exists some $0 \leq j \leq s-1$ such that  $G R_i=G \times R_i$ for every $0 \leq i \leq j$, but $G R_{j+1} \neq G \times R_{j+1}$. In particular, $G R_j=G \times R_j$.

Since $R_j$ is solvable, $R_j\cap G=1$ and $GR_j/R_j\cong G$. By the simplicity of $G$, we have $GR_j/R_j\cap R_{j+1}/R_j=1$, and the conjugation action of $GR_j/R_j$ on $R_{j+1}/R_j$ is either trivial, or faithful. Suppose the action is trivial. Then $GR_{j+1}/R_j=(GR_j/R_j)(R_{j+1}/R_j)=GR_j/R_j\times R_{j+1}/R_j$, implying that $G R_j \unlhd G R_{j+1}$. Noting that $G$ is characteristic in $G R_j$ as $G R_j=G \times R_j$, we have $G\unlhd G R_{j+1}$, and so $G R_{j+1}=G \times R_{j+1}$, a contradiction. It follows that $GR_j/R_j\cong G$ has a faithful conjugation action $R_{j+1} / R_j$. Since $R_{j+1} / R_j\cong \mz_r^d$ for some prime $r$ and positive integer $d$, we have $G \lesssim \GL(d, r)$, and since $G$ is non-abelian simple, we have $d\geq 2$. Clearly, $r^d\ \mid |R|$. This completes the proof of part~(1).

\medskip

To prove part~(2), assume $7\bigm| |R|$. Since $G$ is regular, Proposition~\ref{Frattini} implies that $B=GB_v$ for $v\in V(\Gamma)$ and $|B|=|G||B_v|$. Recall that $|B|=|G||R|$. Then $|R|=|B_v|$ and hence $7\bigm| |B_v|$, that is, $B$ is arc-transitive. Since $7^2  \nmid  |B_v|$, we have $7^2\nmid |R|$, and so a Sylow $7$-subgroup, say $P$, of $R$ has order $7$, that is, $P\cong\mz_7$.

Now we claim $P\unlhd B$. Suppose for a contradiction that $P\ntrianglelefteq B$. Since $R\unlhd B$ is soluble, $B$ has a series of normal subgroups: $$1\unlhd R_1\unlhd R_2\unlhd \cdots\unlhd R_t=R\unlhd B,$$
such that $R_i\unlhd B$ for every $1\leq i\leq t$ and $R_2/R_1\cong\mz_{7}$. Note that $G\cap R_1=1$ and $GR_1/R_1\cong G$. The conjugation action of $GR_1/R_1$ on $R_2/R_1$ is trivial, for otherwise $G\lesssim\Aut(R_2/R_1)=\mz_6$. And so $GR_1/R_1$ commutes with $R_2/R_1$ pointwise. Since $R_2/R_1\cap GR_1/R_1=1$, we have $GR_2/R_1=GR_1/R_1\times R_2/R_1$, implying $GR_1\unlhd GR_2$. By the regularity of $G$, we obtain $|(GR_2)_v|=|R_2|$ and $|(GR_1)_v|=|R_1|$ for $v\in V(\Gamma)$. Then $GR_2$ is arc-transitive as $7\bigm| |R_2|$.

Note that $(GR_2)_v$ is primitive on $\Gamma(v)$ as $\Gamma$ has valency $7$. If $R_1\not=1$, then  $GR_1$ is arc-transitive as $|(GR_1)_v|=|R_1|$ and $GR_1\unlhd GR_2$. It follows $7 \bigm|  |R_1|$, and since $R_2/R_1\cong\mz_{7}$, we have $7^2 \bigm|  |R|$, a contradiction. Thus, $R_1=1$, and so $R_2$ is a normal Sylow $7$-subgroup of $R$. This yields that $P=R_2\unlhd B$, as claimed.

Let $C=C_{B}(P)$. Since $P\unlhd B$, we have $P\leq C$ and $C\unlhd B$. Then the conjugation action of $G$ on $P$ is trivial, for otherwise $G\leq\Aut(P)\cong\mz_6$, contradicting the simplicity of $G$. Thus, $GP=G\times P$, and $G\leq C$. Furthermore, $C=C\cap B=C\cap RG=(C\cap R)G$. Noting that $P$ is a Sylow $7$-subgroup of $C\cap R$, we may write $C\cap R=P\times L$, where $L$ is a Hall $7'$-subgroup of $C\cap R$, and then $C=(P\times L)G=P\times LG$, implying $LG\unlhd C$.

Since $G\times P\leq C$, $C$ is arc-transitive. If $L\neq1$, then $LG$ is arc-transitive becasue  $|(LG)_v|=|LG|/|G|=|L|$ and $LG\unlhd C$. It follows that $7\bigm| |(LG)_v|$ and so $7\bigm| |L|$, a contradiction. Therefore, $L=1$ and $C=P\times G$. In particular, $G$ is characteristic $C$, and since $C\unlhd B$, we have $G\unlhd B$, yielding $B=RG=R\times G$. This completes the proof.
\end{proof}

\begin{lemma}\label{lm:condi3}
Under Assumption, let $G$ be regular on $V(\Gamma)$ and $R$ the radical $X$. Then $R$ has at least three orbits on $V(\Gamma)$, and $RG=R\times G$.
\end{lemma}

\begin{proof}
Set $B=RG$ and let $v\in V(\Gamma)$. Since $R$ is solvable, $G\cap R=1$, $|B|=|R||G|$, and $G\cong B/R$. Since $G$ is regular, Proposition~\ref{Frattini} implies $B=GB_v$, and so $|B_v|=|B|/|G|=|R|$.

First we claim that $R$ has at least three orbits on $V(\Gamma)$. Suppose that $R$ has two orbits. Since $R\unlhd X$ and $X$ is arc-transitive, $\Gamma$ is bipartite and $G$ has a normal subgroup of index $2$, contradicting the simplicity of $G$.

Suppose that $R$ is transitive on $V(\Gamma)$. If $B=R \times G$, by Propostiton~\ref{pro:tran-centra} we have that both $R$ and $G$ are regular and $R\cong G$, contradicting the simplicity of $G$.

Thus, we may assume $B\not=R \times G$. By Lemmma~\ref{lm:condi}~(1), $G\lesssim \GL(d,r)$ for some prime $r$ and integer $d \geq 2$, with $r^{d}\bigm||R|$. Since $R$ is transitive, we have $B=RB_v$, and hence $G\cong B/R=RB_v/R\cong B_v/R_v$. By Proposition~\ref{pro:vertex-stab}, $G=\A_5,\A_6,\A_7$ or $\PSL(3,2)$.

Suppose that $R$ is regular on $V(\Gamma)$. Then $|R|=|G|=|V(\Gamma)|$. If $G\in\{\A_5, \A_{6}$, $\A_{7}$\}, then
$|R| \bigm| 2^{3}\cdot3^{2}\cdot5\cdot7$, and since $\GL(2,3)$ is solvable, we have $G\lesssim \GL(3,2)$, which is impossible by
Atlas~\cite{Atlas}. It follows that $G=\PSL(3,2)$, and so $|R|=|G|=|\PSL(3,2)|=2^3\cdot3\cdot7$. This implies $7\bigm||R|$, and by Lemma~\ref{lm:condi}~(2), $B=R \times G$, a contradiction.

Suppose that $R$ is not regular on $V(\Gamma)$.  Since $R\unlhd X$, $R$ is arc-transitive and hence $B$ is arc-transitive. In particular, $7\nmid  |B_v/R_v|$. Since $G\cong B_v/R_v$, we have $G=\A_5$ or $\A_6$, which is impossible by Proposition~\ref{pro:vertex-stab}.

The claim follows, that is, $R$ has at least three orbits.

\medskip

Now we prove $B=R\times G$. This is clearly true for $R=1$. We assume $R\not=1$.
By Proposition~\ref{pro:quotientgraph}, $R$ is semiregular on $V(\Gamma)$ and the quotient graph $\Gamma_{R}$ is a connected $7$-valent $X/R$-arc-transitive graph, where $R$ is the kernel of $X$ on $V(\Gamma_R)$ and $X/R \leq\Aut(\Gamma_{R}$). Furthermore, $G\cong B/R$ is vertex-transitive on $V(\Gamma_{R})$ and $(B/R)_{\overline{v}}\cong B_v$, where $\overline{v}$ is the orbit of $R$ on $V(\Gamma)$ containing $v$. Since $G$ is regular, $|V(\Gamma_R)|=|G|/|R|$, and then $G\cong GR/R=B/R\leq X/R$, $|B_v|=|(B/R)_{\overline{v}}|=|B/R|/(|G|/|R|)=|B|/|G|=|R|$.

Since $R$ is the radical of $X$, $X/R$ has a  trivial radical. By Lemma~\ref{lm:trival radical}, either
\begin{itemize}
  \item [\rm (a)]$B/R \unlhd X/R$, or
  \item [\rm (b)]$X/R$ is almost simple with $H:=B/R\leq T:=\soc(X/R)$, and $\Gamma_R$ is $T$-arc-transitive. Furthermore, $(T,H)\cong(\A_{8}, \PSL(3,2))$, $(\M_{24}, \M_{23})$ or $(\A_{n},\A_{n-1})$ with $n\geq 7$ and $n\bigm|2^{24}\cdot3^4\cdot5^2\cdot7$.
\end{itemize}

For case~(a), we have $B\unlhd X$. Since $R\not=1$, we have $B_v\not=1$, and $\Gamma$ is $B$-arc-transitive as $X$ is arc-transitive. Thus, $7\bigm| |B_v|$, and since $|B_v|=|R|$,
Lemma~\ref{lm:condi}~(2) implies $B=G\times R$, as required.

For case~(b), we have $G\cong H=B/R\ntrianglelefteq X/R$. To finish the proof, we suppose for a contradiction that $B\neq G\times R$. By Lemma~\ref{lm:condi}~(1), $G\lesssim \GL(d,r)$ for some prime $r$ and integer $d\geq2$, with $r^{d}\bigm|  |R|$. Since $|R|\bigm|  |G|$, we have $r^d\bigm|  |G|$, and since $|R|=|B_v|$, from Proposition~\ref{pro:vertex-stab} we have  $|R|\bigm|  2^{24}\cdot3^4\cdot5^2\cdot7$. In particular, $r=7,5,3$, or $2$. Since $\GL(2,5)$ contains no nonabelian simple subgroup, and  $\GL(d,r)$ is insolvable, we have either $r=3$ and $3\leq d\leq 4$, or $r=2$ and $3\leq d\leq 24$.

Let $r=3$ and $3\leq d\leq 4$. If $d=3$ then by Atlas~\cite{Atlas}, $\SL(3,3)=\PSL(3,3)$ is the unique nonabelian simple subgroup of $\GL(3,3)$, implying $G=\PSL(3,3)$. However, $\PSL(3,3)$ is not listed in (b) as $H\cong G$. If $d=4$ then again by Atlas~\cite{Atlas}, $G=\PSL(3,3),\PSU(4,2),\A_6$ or $\A_5$, and they are all not  listed in  (b), a contradiction.

Let $r=2$ and $3\leq d\leq 24$. Since $\Gamma_{R}$ is $T$-arc-transitive, we have
$|T|=|V(\Gamma_{R})||T_{\overline{v}}|=|G|/|R|\cdot|T_{\overline{v}}|$, and hence $7$ is a divisor of $|T_{\overline{v}}|=|T||R|/|G|$. If $7\nmid |T|/|G|$, then $7\ |\ |R|$, and $B=G\times R$ by Lemma~\ref{lm:condi}~(2), a contradiction. Thus, $7\ |\ |T|/|G|=|T|/|H|$, and by the list (b), $(T,H)\cong (\A_{n},\A_{n-1})$ with $7\ |\ n\ |\ 2^{24}\cdot3^4\cdot5^2\cdot7$.


Suppose $(T,H)\cong (\A_{n},\A_{n-1})$ with $7\ |\ n\ |\ 2^{24}\cdot3^4\cdot5^2\cdot7$. The above paragraph implies that $n\not=7$. It follows that $n=7\ell$ for some $\ell\geq 2$.
By Proposition~\ref{pro:lessGL}, $(n-1)-2\leq d\leq 24$, and hence $n=14$ or $21$, that is, $H=\A_{13}$ or $\A_{20}$. For $H=\A_{13}$, we have $|H|_2=|G|_2=2^9$, and so $d\leq 9$. On the other hand, Proposition~\ref{pro:lessGL} implies that $d\geq R(\A_{13})=13-2=11$, a contradiction. For $H=\A_{20}$, similarly we have $|H|_2=|G|_2=2^{17}$, and so $d\leq 17$, which is impossible because Proposition~\ref{pro:lessGL} implies that $d\geq R(\A_{20})=18$.
\end{proof}

\begin{lemma}\label{lm:condi5}
Under Assumption, let $G$ be regular on $V(\Gamma)$ and $R$ the radical of $X$. Then $G\unlhd X$ or $\soc(X/R)=(T\times R)/R$ for a non-abelian simple group $T$. In particular, $T\unlhd X$ is arc-transitive on $\Gamma$ with $G<T$ and $(T,G)$=$(\A_{n},\A_{n-1})$ with $n=7$, $3\cdot 7$, $3^2\cdot 7$, $2^2\cdot 3\cdot 7$, $2^3\cdot3\cdot7$, $2^3\cdot3^2\cdot5\cdot7$, $2^4\cdot3^2\cdot5\cdot7$, $2^6\cdot3\cdot7$, $2^7\cdot3\cdot7$,
 $2^6\cdot3^2\cdot7$, $2^6\cdot3^4\cdot5^2\cdot7$, $2^8\cdot3^4\cdot5^2\cdot7$, $2^7\cdot3^4\cdot5^2\cdot7$, $2^{10}\cdot3^2\cdot7$, $2^{24}\cdot3^2\cdot7$.
\end{lemma}

\begin{proof}
Set $B=RG$. By Lemma~\ref{lm:condi3}, $B=R \times G$ and $R$ has at least three orbits on $V(\Gamma)$. By Proposition~\ref{pro:quotientgraph},
$R$ is semiregular on $V(\Gamma)$, and the quotient graph $\Gamma_{R}$ is a connected $7$-valent $X/R$-arc-transitive with $X/R \leq \Aut(\Gamma_{R}$). Moreover, $G\cong B/R$ is vertex-transitive on $V(\Gamma_{R})$ and $G$ is characteristic in $B$.

To prove the lemma, assume $G\ntrianglelefteq X$. Since $G$ is characteristic in $B$, we have $B\ntrianglelefteq X$, and hence $G\cong B/R \ntrianglelefteq X/R$. Since $R$ is the largest solvable normal subgroup of $X$, $X/R$ has trivial radical. By Lemma~\ref{lm:trival radical}, $X/R$ is almost simple with $G\cong B/R<\mathrm{soc}(X/R)$, $\Gamma_R$ is soc$(X/R)$-arc-transitive, and  $(\mathrm{soc}(X/R),B/R)=(\A_{8}, \PSL(3,2))$, $(\M_{24}, \M_{23})$, or $(\A_{n},\A_{n-1})$ with $n\geq 7$ and $n\bigm|2^{24}\cdot3^4\cdot5^2\cdot7$.

Set soc$(X/R)=I/R$ and $C=C_{I}(R)$. Then $I/R$ is non-abelian simple because $X/R$ is almost simple. Since $G\cong B/R<I/R$, we have $|G|<|I/R|$ and $B<I$, and since $B=R\times G$, we have $G\leq C$. Clearly, $C\unlhd I$ and $C\cap R=Z(R)\leq Z(C)$. Since $G\cong (R\times G)/R\leq CR/R\unlhd I/R$, the simplicity of $I/R$ implies that $I=CR$.

Thus, $C/(C\cap R)\cong CR/R\cong I/R$, and since $C\cap R\leq Z(C)$, the simplicity of $C/(C\cap R)$ implies that $C\cap R=Z(C)$ and $C/Z(C)\cong I/R$. Again by the simplicity of $I/R$, we have $C'/C'\cap Z(C)\cong C'Z(C)/Z(C)=(C/Z(C))'\cong (I/R)'\cong I/R\cong C/Z(C)$, which implies that $Z(C')=C'\cap Z(C)$, $C'/Z(C')\cong I/R$ and $C=C'Z(C)$. Furthermore, $C'=(C'Z(C))'= C''$ and $C'$ is a covering group of $I/R$. Consequently, $Z(C')$ is a quotient of Mult($I/R$).

Recall that $|G|<|I/R|\leq|C'|$. Since $G\leq C$ and $C/C'$ is abelian, we have $G<C'\unlhd I$ and hence $(C')_v\not=1$ and $I_v\not=1$.
Note that $\Gamma$ has valency $7$ and $X_v$ is primitive on $\Gamma(v)$. Then $\Gamma$ is both $I$-arc-transitive and $C'$-arc-transitive. Clearly, $G\times Z(C')\leq C'$.

By \cite[Theorem~5.1.4]{K-Lie}, Mult($\M_{24})=1$, Mult($\A_{7}$)=$\mz_{6}$ and Mult($\A_{n}$)=$\mz_{2}$ for $n \geq 8$. Recall that $I/R=\mathrm{soc}(X/R)=\M_{24}$, or $\A_{n}$ with $n\geq 7$ and $n\bigm|2^{24}\cdot3^4\cdot5^2\cdot7$, and $Z(C')$ is a quotient of Mult($I/R$).

Suppose $|Z(C')|>2$. Then
$(I/R,B/R)=(\A_7,\A_6)$, and so $Z(C')=\mz_3$ or $\mz_6$, which implies that $C'=3.\A_7$ or $6.\A_7$.
For $C'=3.\A_7$, $G\times Z(C')\cong \A_6\times \mz_3$ is a subgroup of index $7$ in $C'$.
We construct the group $C'=3.\A_7$ in Magma~\cite{Magma} with the following command
\begin{verbatim}
G:=pCover(Alt(7), FPGroup(Alt(7)), 3); X:=CosetImage(G,sub<G|>);
Hs:=Subgroups(X:IndexEqual:=7); [IsPerfect(i`subgroup): i in Hs];
\end{verbatim}
and it shows that $C'$ has only one conjugacy class of subgroups of index $7$ and these subgroups are perfect.
Therefore,  $C'$ has no subgroup not isomorphic to  $\A_6 \times \mz_3$, a contradiction.
For $C'=6.\A_7$, $G\times Z(C')\cong \A_6\times \mz_6$ is a subgroup of index $7$ in $C'$.
Then $C'$ has a quotient group isomorphic to $2.\A_7$ that has a subgroup $\mz_2\times \A_6$, which contradicts that $2.\A_{7}$ has a unique subgroup $2.\A_{6}$ of index $7$ in Proposition~\ref{pro:shurmult}.

Suppose $Z(C')=\mz_{2}$.
Then $(I/R,B/R)=(\A_{8},\PSL(3,2))$ or $(\A_n, \A_{n-1})$ with $n \geq 8$.
For $(I/R,B/R)$=$(\A_{n}, \A_{n-1})$ with $n \geq 8$, by  Proposition~\ref{pro:shurmult}, all subgroups of index $n$ of $C'$ is $2.\A_{n-1}$, contradicting that $G \times Z(C')\cong \A_{n-1} \times \mz_{2}$ is a subgroup of index $n$ of $C'$. Now let $(I/R,B/R)=(\A_{8}, \PSL(3,2))$. Note that
$|I/R|/|B/R|=|I/R|/|G|=2^3\cdot3 \cdot 5$, and  $C'$ is arc-transitive. Since $|C'_v|=|C'|/|G|=|Z(C')||I/R|/|G|$, we have $|C'_v|=2^4\cdot3\cdot5$, of which is impossible by Proposition~\ref{pro:vertex-stab}.

The above contradictions force $Z(C')=1$. This means that $C'\cong I/R$ is a non-abelian simple group. Note that $|I|=|I/R||R|=|C'||R|$, $G<C'\unlhd I$ and $C' \cap R=1$. Then $I=C'\times R$ and $\soc(X/R)=I/R=(C'\times R)/R$. Thus, $C'\times R\unlhd X$ and so $C'\unlhd X$ because $C'$ is characteristic in $C'\times R$. Since $G<C'$ and $C'$ is arc-transitive,
by taking $C'=X$ in Lemma~\ref{lm:trival radical}, we have $(C',G)=(\A_{n},\A_{n-1})$ with
$n=7$, $3\cdot 7$, $3^2\cdot 7$, $2^2\cdot 3\cdot 7$, $2^3\cdot3\cdot7$, $2^3\cdot3^2\cdot5\cdot7$, $2^4\cdot3^2\cdot5\cdot7$, $2^6\cdot3\cdot7$, $2^7\cdot3\cdot7$,
 $2^6\cdot3^2\cdot7$, $2^6\cdot3^4\cdot5^2\cdot7$, $2^8\cdot3^4\cdot5^2\cdot7$, $2^7\cdot3^4\cdot5^2\cdot7$, $2^{10}\cdot3^2\cdot7$, $2^{24}\cdot3^2\cdot7$. This completes the proof by taking $T=C'$.
\end{proof}

\bigskip
\noindent{\bf Proof of Theorem~\ref{th:main-vald}} Let $G$ be a non-abelian simple group and let $\Gamma$ be a connected $7$-valent symmetric Cayley graph on $G$. Set $A=\Aut(\Gamma)$ and $R=\mathrm{rad}(A)$. Note that $G\cong R(G)\leq A$ and $R(G)$ is regular on $V(\Gamma)$.
By taking $A=X$ and $R(G)=G$ in Lemma~\ref{lm:condi5}, either $R(G)\unlhd A$, that is, $\Gamma$ is normal, or $A$ contains an arc-transitive non-abelian simple normal subgroup $T$ such that $R(G)<T$ and $(T,R(G))=(\A_{n},\A_{n-1})$ with $n=7$, $3\cdot 7$, $3^2\cdot 7$, $2^2\cdot 3\cdot 7$, $2^3\cdot3\cdot7$, $2^3\cdot3^2\cdot5\cdot7$, $2^4\cdot3^2\cdot5\cdot7$, $2^6\cdot3\cdot7$, $2^7\cdot3\cdot7$,
 $2^6\cdot3^2\cdot7$, $2^6\cdot3^4\cdot5^2\cdot7$, $2^7\cdot3^4\cdot5^2\cdot7$, $2^8\cdot3^4\cdot5^2\cdot7$, $2^{10}\cdot3^2\cdot7$, $2^{24}\cdot3^2\cdot7$. Moreover, $\soc(A/R)=(T\times R)/R$. This completes the proof.\qed

\section*{Acknowledgments}   
This work was partially supported by the National Natural Science Foundation of China (12331013, 12311530692, 12271024, 12161141005, 12071023, 12301461) and the 111 Project of China (B16002).
 	
\section*{Conflict of interest} 
The authors declare they have no financial interests.
	
\section*{Availability of data and materials} 
Data sharing not applicable to this article as no datasets were generated or analysed during the current study.

\end{document}